\documentclass{article}

\usepackage{amsmath}
\usepackage{amsfonts}
\usepackage{amssymb}
\usepackage{amsthm}
\usepackage{enumerate}
\usepackage{bbm}
\usepackage{hyperref}

\newtheorem{theorem}{Theorem}[section]
\newtheorem{lemma}[theorem]{Lemma}
\newtheorem{claim}[theorem]{Claim}

\newtheorem{corollary}[theorem]{Corollary}

\newtheorem{conjecture}[theorem]{Conjecture}

\newcommand{\suchthat}{\;\ifnum\currentgrouptype=16 \middle\fi|\;}

\newcommand{\R}{\mathbb{R}}

\DeclareMathOperator{\disc}{disc}

\newcommand{\E}{{\rm I\kern-.3em E}}
\newcommand{\Var}{\mathrm{Var}}

\begin{document}
\title{A spectral bound on hypergraph discrepancy}
\author{
Aditya Potukuchi \thanks{Department of Computer Science,  Rutgers University. {\tt aditya.potukuchi@cs.rutgers.edu}. Research supported in part by NSF grant 
CCF-1514164} 
}

\maketitle

\begin{abstract}
Let $\mathcal{H}$ be a $t$-regular hypergraph on $n$ vertices and $m$ edges. Let $M$ be the $m \times n$ incidence matrix of $\mathcal{H}$ and let us denote $\lambda =\max_{v \perp \overline{1},\|v\| = 1} \|Mv\|$. We show that the discrepancy of $\mathcal{H}$ is $O(\sqrt{t} + \lambda)$. As a corollary, this gives us that for every $t$, the discrepancy of a random $t$-regular hypergraph with $n$ vertices and $m \geq n$ edges is almost surely $O(\sqrt{t})$ as $n$ grows. The proof also gives a polynomial time algorithm that takes a hypergraph as input and outputs a coloring with the above guarantee.
\end{abstract}

\section{Introduction}

The main aim of this paper is to give a spectral condition that is sufficient for the discrepancy of a regular hypergraph to be small. This is proved via the partial coloring approach while using some combinatorial properties of the hypergraph that are given by this spectral condition. This immediately implies, via an old proof technique of Kahn and Szemer\'{e}di, that for every $t$, the discrepancy of a random $t$-regular hypergraph on $n$ vertices and $m \geq n$ edges is almost surely $O\left(\sqrt{t}\right)$. Previously, a result of this form was proved by Ezra and Lovett~\cite{EL15} who show that the discrepancy of a random $t$-regular hypergraph on $n$ vertices and $m\geq n$ edges is $O(\sqrt{t \log t})$ almost surely as $t$ grows. More recently, Bansal and Meka~\cite{BM19} showed that for random $t$-regular hypergraphs on $n$ vertices and $m$ edges, the discrepancy is $O\left( \sqrt{t} \right)$ almost surely provided $t = \Omega\left((\log \log m)^2\right)$. To state our result formally, we make some definitions. \\

Let $\mathcal{H} = (V,E)$ be a hypergraph, with $V$ as the set of vertices, and $E \subseteq 2^{V}$ as the set of (hyper)edges. Let $\mathcal{X} = \{\chi : V \rightarrow \{\pm 1\}\}$, be the set of $\pm 1$ colorings of $V$, and for $\chi \in \mathcal{X}$, and $e \in E$, denote $\chi(e) := \sum_{v \in e}\chi(v)$. The discrepancy of $\mathcal{H}$, denoted by $\disc(\mathcal{H})$ is defined as:
\[
\disc(\mathcal{H}) := \min_{\chi \in \mathcal{X}}\max_{e \in E} |\chi(e)|.
\]

We call a hypergraph \emph{$t$-regular} if every vertex is present in exactly $t$ hyperedges. These will be the main focus of this paper. For a hypergraph $\mathcal{H}$, let $M = M(\mathcal{H})$ be the $|E| \times |V|$ incidence matrix of $\mathcal{H}$, i.e.,  $M$ has rows indexed by $E$, columns indexed by $V$, and entries are $M(e,v) = 1$ if $v \in e$ and $0$ otherwise. We will use $\|\cdot \|$ to denote the Euclidean norm throughout the paper. Our main result is the following:

\begin{theorem}
\label{thm:main}
 Let $\mathcal{H}$ be a $t$-regular hypergraph on $n$ vertices and $m$ edges with $M$ as its incidence and let $\lambda = \max_{v \perp \overline{1}, \|v\| = 1}\|Mv\|$. Then
\[
\disc(\mathcal{H}) = O\left(\sqrt{t} + \lambda\right).
\]
Moreover, there is an $\tilde{O}((\max\{n,m\})^{7})$ time algorithm that takes the hypergraph $\mathcal{H}$ as input and outputs the coloring with the above guarantee.
\end{theorem}

\subsection{Background}

The study of hypergraph discrepancy, which seems to have been first defined in a paper of Beck~\cite{BEC81}, has led to some very interesting results with diverse applications (see, for example~\cite{MAT98},~\cite{CHA00}). One of the most interesting open problems in discrepancy theory is what is commonly known as the Beck-Fiala conjecture, regarding the discrepancy of general $t$-regular hypergraphs. 

\begin{conjecture}[Beck-Fiala conjecture]
\label{conj:BF}
For a $t$-regular hypergraph $\mathcal{H}$, we have
\[
\disc(\mathcal{H}) = O(\sqrt{t}).
\]
\end{conjecture}

Although this conjecture is usually stated for \emph{bounded degree} hypergraphs (as opposed to regular ones), this is not really an issue. One can always add hyperedges containing just
a single vertex and make it regular, which increases the discrepancy of the original
hypergraph by at most one. Beck and Fiala~\cite{BF81} also proved that for any $t$-regular hypergraph $\mathcal{H}$,
\[
\disc(\mathcal{H}) \leq 2t-1.
\] 

This is more commonly known as the Beck-Fiala theorem. Essentially the same proof can be done a bit more carefully to get a bound of $2t - 3$ (see~\cite{BH97}). Given Conjecture~\ref{conj:BF}, it is perhaps surprising that the best upper bound, due to Bukh~\cite{BUK16}, is ``stuck at" $2t - \log^*t$ for large enough $t$.\\

It is possible that one of the reasons that the discrepancy upper bounds are so far away from the conjectured bound (assuming it's true) is our inability to handle many `large' hyperedges. Indeed, if one is offered the restriction that each hyperedge is also of size $O(t)$ (regular and `almost uniform'), then a folklore argument using the Lov\'{a}sz Local Lemma shows that the discrepancy is bounded by $O(\sqrt{t \log t})$. The proof of Theorem~\ref{thm:main} also relies on being able to avoid dealing with large edges (which are few, if any, in number).

\subsection{Discrepancy in random settings}

Motivated by the long-standing open problem of bounding discrepancy of general $t$-regular hypergraphs, Ezra and Lovett~\cite{EL15} initiated the study of discrepancy of \emph{random} $t$-regular hypergraphs. By random $t$-regular hypergraph, we mean the hypergraph sampled by the following procedure: We fix $n$ vertices $V$ and $m$ (initially empty) hyperedges $E$. Each vertex in $V$ chooses $t$ (distinct) hyperedges in $E$ uniformly and independently to be a part of. They showed that if $m \geq n$, then the discrepancy of such a hypergraph is almost surely $O(\sqrt{t \log t})$ as $t$ grows. The proof idea is the following: First observe that most of the hyperedges have size $O(t)$. For the remaining large edges, one can delete one vertex from every hyperedge and make them pairwise disjoint. This allows one to apply a folklore Lov\'{a}sz Local Lemma based argument, but with a slight modification which makes sure that the large edges have discrepancy at most $2$. More recently, Bansal and Meka~\cite{BM19} reduced the discrepancy bound to $O(\sqrt{t})$ almost surely as long as $t = \Omega\left((\log \log n)^2\right)$ for all $m$ and $n$. A corollary of Theorem~\ref{thm:main} states that one can get the bound of $O(\sqrt{t})$ for every (not necessarily growing) $t  = t(n)$ as $n$ grows and $m \geq n$. More formally, 

\begin{corollary}
\label{cor:ranreg}
There is an absolute constant $C >0$ such that the following holds:
Let $\mathcal{H}_t$ be a random $t$-regular hypergraph on $n$ vertices and $m \geq n$ hyperedges where $t = o(\sqrt{m})$. Then,
\[
\mathbb{P}\left(\disc(\mathcal{H}_t) \leq C\sqrt{t}\right) \geq 1 - o(1)
\]
\end{corollary}

 The theorem that implies Corollary~\ref{cor:ranreg} from Theorem~\ref{thm:main} is the following:

\begin{theorem}
\label{thm:norm}
 Let $M$ be the incidence matrix of a random $t$-regular set system on $n$ vertices, where $t = o(\sqrt{m})$, and $m \geq n$ edges. Then with probability at least $1 - n^{\Omega(1)}$,
 \[
  \max_{v \perp \overline{1},\|v\| = 1}\|Mv\| = O\left(\sqrt{t}\right).
 \] 
\end{theorem}

A couple of remarks here: First, observe that it suffices to prove Theorem~\ref{thm:norm} for $m = n$. Indeed, let $M$ and $N$ be random $m \times m$ and $m \times n$ random matrices (m > n) respectively distributed by choosing $t$ random $1$'s in each column independently. Notice that the distribution of $N$ is exactly the same as that of the first $n$ columns of $M$. Then, setting $M_n$ to be the matrix consisting of the first $n$ columns of $M$, we observe that $\lambda(M_n) \leq \lambda(M)$. Second, we point out that $t = o(\sqrt{m})$ is just a limitation of the proof technique in~\cite{FKS} (also see~\cite{BFSU98}) that we use to prove this theorem. Although we believe that Theorem~\ref{thm:norm} should hold for all $t < m$, we do not make any attempt to verify this, especially since the result of Bansal and Meka~\cite{BM19} already takes care of the discrepancy of random hypergraphs in this case. \\

Although many variations of Theorem~\ref{thm:norm} are known and standard, one needs to verify it for our setting too. It should come as no surprise that the proof follows that of Kahn and Szemer\'{e}di's \footnote{\cite{FKS} is combination of two papers that prove the same result upto a constant factor: one by Friedman using the so-called trace method, and the other by Kahn and Szemer\'{e}di using a more combinatorial approach which is flexible enough to be easily adapted here.} in~\cite{FKS}, which is postponed to the Appendix~\ref{sec:norm}.

\subsection{The partial coloring approach}

Most of the bounds and algorithms on hypergraph discrepancy proceed via a \emph{partial coloring approach}. In general, a partial coloring approach~\cite{BEC81} works by coloring a fraction of the (still uncolored) vertices in each step, while ensuring that no edge has discrepancy more than the desired bound. Perhaps the most famous successful application of this is Spencer's celebrated `six standard deviations' result~\cite{SPE85}, which gives a bound of $6\sqrt{n}$ for any hypergraph on $n$ vertices and $n$ edges. The original proof of Spencer was not algorithmic, i.e., it did not give an obvious way to take as input a hypergraph on $n$ vertices and $n$ edges, and efficiently output a coloring that achieves discrepancy $O(\sqrt{n})$. In fact, Alon and Spencer(\cite{AS00}, $\S 14.5$) suggested that such an algorithm is not possible. However, this was shown to be incorrect by Bansal~\cite{BAN10} who showed an efficient algorithm to do the same task. However, the analysis of this algorithm still relied on the (non-algorithmic) discrepancy bound of $6 \sqrt{n}$. Later, Lovett and Meka~\cite{LM15} gave a `truly constructive' proof of the fact that the discrepancy is $O(\sqrt{n})$. This proof did not rely on any existing discrepancy bounds and the novel and simple analysis proved to be extremely influential. The proof of Theorem~\ref{thm:main} will rely on a somewhat technical feature of the main partial coloring from this work. More recently, a result due to Rothvoss~\cite{ROT17} gives a simpler proof of the same $O(\sqrt{n})$ bound, which is also constructive, and more general.

\subsection{Proof sketch}

The proof of Theorem~\ref{thm:main} is proved via the aforementioned partial coloring approach. The main source of inspiration is a later paper of Spencer~\cite{SPE88}, which computes the discrepancy of the projective plane (i.e., the hypergraph where the vertices are the points and the hyperedges are the lines of $\text{PG}(2,q)$) upto a constant factor. A more general bound was also obtained by Matou\v{s}ek~\cite{MAT95}, who upper bounds the discrepancy of set systems of bounded VC-dimension (note that the projective plane has VC-dimension $2$). \\

We also use the aforementioned result of Lovett and Meka~\cite{LM15} heavily, in particular, the partial coloring theorem. Informally, this says that one can `color' roughtly an $\alpha$ fraction of the hypergraph with real numbers in $[-1,1]$ so that (1) at least half the vertices get colors $1$ or $-1$ and (2) every edge $e$ has discrepancy $O(\sqrt{e})$. We now sketch the proof.\\

Consider the following `dream approach' using partial coloring: In every step, one colors an $\alpha$ fraction of vertices. Suppose that at the start, every edge has size $O(t)$ and that each step of partial coloring colors exactly an $\alpha$ fraction of the remaining uncolored vertices (i.e., these vertices are colored from $\{-1,1\}$). Then the discrepancy of an edge $e$ is at most $O\left(\sum_{i}\sqrt{\alpha^i|e|}\right) = O(\sqrt{t})$. Of course, this is too much to hope for, since some edges can potentially be large, and  more importantly, there is no guarantee on how much of each edge gets colored in this partial coloring procedure. \\

This is precisely where the spectral condition on $M$ saves us. One can establish standard combinatorial `pseudorandomness' properties of $\mathcal{H}$ in terms of $\lambda$. In particular, if $\lambda$ is small, then an $\alpha$ fraction of $V(\mathcal{H})$ take up an $\alpha$ fraction of most edges. This means, intuitively, that in the partial coloring approach, if one colors an $\alpha$ fraction of the vertices, then most of the edge sizes will have also reduced by an $\alpha$ fraction. The partial coloring method of Lovett and Meka (and, curiously, none of the older ones) also allows one to color in such a way that $\Omega(n)$ edges can be made to have discrepancy zero in each step. This allows one to maintain that in every round of the partial coloring, the edges that don't behave according to the `dream approach', i.e., those that are too large (i.e., $\Omega(t)$) or don't reduce by an $\alpha$ fraction can be made to have discrepancy \emph{zero} in the next step. Thus, most other edges reduce in size by an $\alpha$ fraction. This lets one not have to deal with the discrepancy of these  `bad' edges until they become small.

\section{Proof of Theorem~\ref{thm:main}}
\label{sec:H1}

\subsection{Preliminaries and notation}

We will need the aforementioned partial coloring theorem due to Lovett and Meka:
\begin{theorem}[\cite{LM15}]
\label{thm:LovettMeka}
Given a family of sets $M_1,\ldots,M_m \subseteq [n]$, a vector $x_0 \in [-1,1]^n$, positive real numbers $c_1,\ldots,c_m$ such that $\sum_{i \in [m]}\exp\left(-c_i^2/16\right) \leq n/16$, and a real number $\delta \in [0,1]$, there is a vector $x \in [-1,1]^n$ such that:

\begin{itemize}
\item[1.] For all $i \in [m]$, $\langle x - x_0, \mathbbm{1}_{M_i}\rangle \leq c_i \sqrt{|M_i|}$.
\item[2.] $|x_i| \geq 1 - \delta$ for at least $n/2$ values of $i$.
\end{itemize}

Moreover, this vector $x$ can be found in $\tilde{O}((m+n)^3 \delta^{-2})$ time.
\end{theorem}

Lovett and Meka initially gave a randomized algorithm for the above. It has since been made deterministic~\cite{LRR17}.

\subsubsection{A technical remark:}
The reason we use the Lovett-Meka partial coloring, as opposed to Beck's partial coloring is not just the algorithmic aspect that the former offers, but also because it also offers the technical condition:
\[
\sum_{i \in [m]}\exp\left(-c_i^2/16\right) \leq n/16.
\] 

This means one can set $\Omega(n)$ edges to have discrepancy $0$. To compare, we first state Beck's partial coloring lemma (for reference, see~\cite{MAT98}):

\begin{theorem}[Beck's partial coloring lemma]
\label{thm:Beck}
Given a family of sets $M_1,\ldots,M_m \subseteq [n]$, and positive real numbers $c_1,\ldots,c_m$ such that $\sum_{i \in [m]}g(c_i) \leq n/5$, where

\[
g(x)=
        \begin{cases}
             e^{-x^2/9} & x > 0.1 \\
             \ln(1/x) &x \leq 0.1
        \end{cases}
\]

there is a vector $x \in \{-1,0,1\}^n$ such that:

\begin{itemize}
\item[1.] For all $i \in [m]$, $\langle x, \mathbbm{1}_{M_i}\rangle \leq c_i \sqrt{|M_i|}$.
\item[2.] $|x_i| = 1$ for at least $n/2$ values of $i$.
\end{itemize}
\end{theorem}

If one ignores the algorithmic aspect, Beck's partial coloring, while assigning vertices to $\{-1,1,0\}$ (instead of $[-1,1]$, thus making it a `partial coloring' in the true sense) only guarantees that $\Omega\left( \frac{n}{\log t}\right)$ edges can be made to have discrepancy $0$. Although~\cite{LM15} did not really need this particular advantage, they do mention that this feature could potentially be useful elsewhere. This seemingly subtle advantage turns out to be crucial in the proof of Theorem~\ref{thm:main}, where we set $\Omega(n)$ edges (that will be called `bad' and `dormant' edges) to have discrepancy $0$.\\

Henceforth, let $V$ and $E$ denote the vertices and edges of our hypergraph respectively. We will need a `pseudorandomness' lemma that informally states that an $\alpha$ fraction of vertices takes up around an $\alpha$ fraction of most edges:
\begin{lemma} 
\label{lem:pseudorandom}
For any $S \subseteq V$ with $|S| = \alpha n$ where $\alpha \in (0,1)$ and a positive real number $K$, there is a subset $E' \subset E$ of size at most $K^{-2} \cdot \alpha n $ such that for every $e \not \in E'$, we have $||e \cap S| - \alpha |e|| \leq  K \lambda$, where $\lambda = \max_{v \perp \overline{\mathbf{1}}, \|v\| = 1}\|Mv\|$.
\end{lemma}

\begin{proof}
Consider a vector $v \in \R^n$ where $v(i) = 1-\alpha$ for $i \in S$ and $-\alpha$ otherwise. Clearly, $v \in \mathbf{1}^{\perp}$ and so 
\begin{equation}
\label{eqn:pseudorandomnorm}
\|Mv\|^2 \leq \lambda^2 \cdot  \|v\|^2 = \lambda^2 \alpha(1 - \alpha)n.
\end{equation}
 On the other hand, $Mv (e) = (1-\alpha)|e \cap S| - \alpha|e \setminus S| = |e \cap S| - \alpha|e|$, and so 
 \begin{equation}
 \label{eqn:pseudorandomnorm2}
 \|Mv\|^2 = \sum_{e}(|e \cap S| - \alpha|e|)^2.
 \end{equation}
 Putting (\ref{eqn:pseudorandomnorm}) and (\ref{eqn:pseudorandomnorm2}) together, we get that there at most $K^{-2}\cdot \alpha n$ edges $e$ such that \linebreak $||e \cap S|  - \alpha|e|| \geq K\lambda$.
\end{proof}

Since this proof is via partial coloring, let us use $i$ to index the steps of the partial coloring. For a partial coloring $\chi:V \rightarrow [-1,1]$, we call the set of vertices $u$ for which $|\chi(u)| < 1$ as \emph{uncolored}. Let us use  $V^i$ to denote the still uncolored vertices at step $i$ and for an edge $e\in E$, let us denote $e^i := e \cap V^i$. In every step, we invoke Theorem~\ref{thm:LovettMeka} setting $\delta = \frac{1}{n}$ to get the partial coloring, so will have $|V^i|\leq 2^{-i}n$. Let $t' := \max\{t, \lambda\}$ \footnote{In fact, we may assume w.l.o.g. that $\lambda \leq t$ and so $t' = t$ since in the other case, the Beck-Fiala Theorem gives us that the discrepancy is $O(t) = O(\lambda)$. However, this is not needed and the techniques here also handle this case with this minor change.}. \\

We call an edge \emph{dormant} at step $i$ if $|e^i| > 100  t'$. Let us call an edge \emph{bad} in step $i$ if $\left||e^i| - 2^{-i}|e|\right|  \geq  10 \lambda$. Edges that are neither dormant not bad are called \emph{good}. Finally, we say that $e$ is \emph{dead} in step $i$ if $|e^i| \leq 100\lambda$. \\

Informally, the roles of these sets are as follows: In the partial coloring step $i$, we ensure that an edge $e$ edges only get nonzero discrepancy if it is good, i.e., if $|e^i|$ is close to what is expected and is not too large. Even dead edges can be good or bad, and we will not distinguish them while coloring the vertices. However, in the analysis we will break the total discrepancy accumulated by $e$ into two parts: Before it is dead and after. The main point is to bound the discrepancy gained before it becomes dead. After it becomes dead, we simply bound the discrepancy incurred since by its remaining size, i.e., at most $100\lambda$.

First, we make two easy observations:

\begin{claim}
\label{claim:dormant}
If $|V^i| = 2^{-i}n$, then at step $i$, the number of dormant edges is at most $\frac{1}{100}2^{-i}n$.
\end{claim}
\begin{proof}
This is just Markov's inequality, using the fact that the average edge size is \linebreak $\frac{|V^i|t}{m} \leq \frac{|V^i|t'}{m}$.
\end{proof}

\begin{claim}
\label{claim:bad}
If $|V^i| = 2^{-i}n$, then at step $i$, the number of bad edges is at most $\frac{1}{100} 2^{-i}n$.
\end{claim}
\begin{proof}
This is by setting $K= 10$ and $\alpha = 2^{-i}$ in Lemma~\ref{lem:pseudorandom}.
\end{proof}

\subsection{Partial coloring using Lemma~\ref{lem:pseudorandom}}

\begin{proof}[Proof of Theorem~\ref{thm:main}]
Setting $V^0 = V$, we proceed by partial coloring that colors exactly half the remaining uncolored vertices at each stage. For a step $i \geq 0$, suppose that $|V^i| = 2^{-i}n$. We will describe a partial coloring given by $\chi_i : V^i \rightarrow [-1,1]$ that colors half the vertices of $V^i$. \\

For $\ell \geq 1$, let $ A_{\ell} := \{e \in E ~|~ |e| \in [100\cdot2^{\ell }t',100\cdot2^{\ell + 1} t')\}$, and $A_0 := \{e \in E ~|~ |e| < 200 t'\}$. Observe that the edges in $A_{\ell}$ for $\ell \geq 1$ are either bad or dormant in steps $i < \ell$. Also observe that $|A_{\ell}| \leq \frac{2^{-\ell}}{100}n$ for $\ell \geq 1$,  Define constants $\{c_e\}_{e \in E}$ as follows: 

\begin{equation*}
c_e = \begin{cases}
4\sqrt{2 \ln\left(\frac{1}{2^{\ell - i}}\right)} &\text{if $e \in A_{\ell}$ for $\ell \geq 1$ is good}\\
4 \sqrt{\ln\left(\frac{200 t'}{2^{-i} |e|}\right)} &\text{if $e \in A_0$ is good}\\
0 & \text{otherwise}.
\end{cases}
\end{equation*}
Let $\mathcal{B} = \mathcal{B}^i$ and $\mathcal{D} = \mathcal{D}^i$ denote the bad and dormant edges respectively. We handle the edges in $A_0$ and $E \setminus A_0$ separately. For edges in $E \setminus A_0$, we have:

\begin{align*}
\sum_{e \in E \setminus A_0}e^{-\frac{c_e^2}{16}} & \leq \sum_{e \in E \setminus (\mathcal{B} \cup \mathcal{D} \cup A_0)}e^{- \frac{c_e^2}{16}}  + |\mathcal{B}| + |\mathcal{D}| \\
& \leq \sum_{1 \leq \ell \leq i}\sum_{e \in A_{\ell}} e^{2\ln(2^{\ell - i})}  + \frac{2^{-i}n}{50} \\
& = \sum_{1 \leq \ell \leq i} |A_{\ell}|2^{2(\ell - i)} + \frac{2^{-i}n}{50} \\
& \leq \frac{n}{100}\sum_{\ell \leq i}2^{-\ell} \cdot 2^{2\ell - 2i}+ \frac{2^{-i}n}{50} \\
& = \frac{2^{-i}n}{100}\sum_{\ell \leq i}2^{\ell - i} + \frac{2^{-i}n}{50} \\
& \leq \frac{2^{-i}n}{25}. \\
\end{align*}

The second inequality above follows from Claim~\ref{claim:dormant} and Claim~\ref{claim:bad}. For the other case, we have 

\[
\sum_{e \in A_0}e^{-\frac{c_e^2}{16}} \leq \sum_{e \in A_0}e^{\ln\left(\frac{2^{-i}|e|}{200t}\right)} = \frac{2^{-i}}{200}\sum_{e \in E}\frac{|e|}{t'} = \frac{2^{-i}n}{200}.
\]

Here we have used the fact that since the hypergraph is $t$-regular, we have $\sum_{e \in E}|e| = nt \leq nt'$. Putting these together, we have

\[
\sum_{e \in E} e^{-\frac{c_e^2}{16}} \leq \frac{2^{-i}n}{200} +  \frac{2^{-i}n}{50} \leq \frac{|V^i|}{20}.
\]

 Therefore, Theorem~\ref{thm:LovettMeka} guarantees that there is a fractional coloring $\chi_i: V^i \rightarrow [-1,1]$ such that
 \begin{itemize}
\item[1.] $|\chi_i(v)| \geq 1 - \frac{1}{n}$ for at least half of $V^i$.
\item[2.] All the bad and dormant edges get discrepancy $0$.
\item[3.] A good and live edge $e$ gets discrepancy at most $c_e\sqrt{|e^i|}$.
\end{itemize}

Finally, we pick an arbitrary subset of all the vertices $v$ such that $|\chi_i(v)| \geq 1 - \frac{1}{n}$ of size exactly $(1/2) \cdot |V^i|$ and round them to the nearest integer. It is easy to see that since every edge has size at most $n$, this rounding, over all the steps of the partial coloring adds discrepancy of at most $1$ for every edge. This completes step $i$ of the partial coloring and we are left with $2^{-(i+1)}n$ uncolored vertices for the next step. \\

For an edge $e$, let $i$ be a round where $e$ had incurred non-zero discrepancy and $e^i$ was not dead. Since only good edges incur nonzero discrepancy, $|e^{i}| = 2^{-i}|e| \pm 10\lambda$. Since $e$ is also not dead at step $i$, we must have that $|e^{i}| \geq 100\lambda$. This gives us that $2^{-i}|e| \geq 90\lambda$ and therefore $(1/2) \cdot 2^{-i}|e| \leq |e_i| \leq 2 \cdot2^{-i}|e|$. So, if $e \in A_{\ell}$ where $\ell \geq 1$, the total discrepancy incurred by $e$ at step $i$ without the rounding step is at most 

\[
4\sqrt{2 \ln(1/2^{j - i})e^{i}} \leq 8\sqrt{200 \ln(1/2^{\ell - i}) \cdot ( 2^{\ell - i})\cdot  t'}.
\]

Here, we have used the fact that $|e| \leq 100\cdot 2^{\ell + 1}t'$. If $e \in A_0$, the discrepancy incurred by $e$ at step $i$ without the rounding is at most

\[
4\sqrt{2 \ln\left(\frac{200t'}{2^{- i}|e|}\right)e^{i}} \leq 8\sqrt{200 \ln(1/2^{ - i}) \cdot ( 2^{- i})\cdot  t'}.
\]

 Therefore, the discrepancy of an edge $e \in A_{\ell}$ for $\ell \geq 0$ until it becomes dead is at most 
\begin{align*}
\sum_{i \geq \ell}8\sqrt{200  \ln(1/2^{\ell - i})\cdot (2^{\ell-i})  \cdot t'} = O(\sqrt{t'}) = O(\sqrt{t} + \lambda).
\end{align*}
Here we have used the fact that $t' = \max\{t,\lambda\}$. Finally, rounding the color of every vertex to its nearest integer increases the discrepancy by at most $1$. When the edge becomes dead, we simply bound its discrepancy by its size $O(\lambda)$.\\

It remains to check that each of the $O(\log n)$ stages of partial coloring can be done in time $\tilde{O}((m+n)^3n^2)$, and the constants $\{c_e\}_{e \in E}$ take $\tilde{O}(mn)$ time to compute at each stage, thus establishing the algorithmic part.
\end{proof}

\section{Conclusion}

We have given an upper bound on $t$-regular hypergraph discrepancy in terms of $t$ and a spectral property of the incidence matrix. However, when one restricts attention to random $t$-regular hypergraphs, the $O(\sqrt{t})$ bound is achieved only when $m = \Omega(n)$.  In the case where $m = o(n)$, one can replace $\lambda$ in Theorem~\ref{thm:main} by $\lambda'$ where
\[
\lambda'(\mathcal{H}) := \max_{\substack{U \subset V \\ |U| = 16m}} ~~\max_{\substack{v \perp \overline{1}, \\\|v\| = 1,\\ \operatorname{supp}(v) \subseteq U}}\|Mv\|
\]
and the proof would remain the same. This is because using the partial coloring theorem (Theorem~\ref{thm:LovettMeka}), one may assign colors to all but at most $16m$ vertices while maintaining that the discrepancy of \emph{every} edge is $0$. However, when $\mathcal{H}$ is a random $t$ regular hypergraph with $n$ vertices and $m = o(n)$ edges, we need not have $\lambda'(\mathcal{H}) = O\left(\sqrt{t}\right)$ (in fact, the guess would be $O(\sqrt{tn/m})$). The problem is that Claim~\ref{claim:smallpairs} (In Appendix~\ref{sec:norm}) does not extend. However, in this regime, we believe that with high probability, the discrepancy is much \emph{lower} than $\sqrt{t}$ (in contrast to $\lambda$ growing). \\

Recently, Franks and Saks~\cite{FS18} showed that for $n = \tilde{\Omega}(m^3)$, the discrepancy is $O(1)$ almost surely. Independently, Hoberg and Rothvoss~\cite{HR19} considered a different model of random hypergraphs with $n$ vertices and $m$ edges and each vertex-edge-incidence is an i.i.d. $\operatorname{Ber}(p)$ random variable. They show that if $n = \tilde{\Omega}(m^2)$, the discrepancy is $O(1)$ almost surely. Both~\cite{FS18} and~\cite{HR19} used similar Fourier analytic techniques inspired by~\cite{KLP12}. Moreover, it was an open question in~\cite{HR19} whether the hypergraph with i.i.d $\operatorname{Ber}(1/2)$ incidences where $n = O(m \log m)$ almost surely has discrepancy $O(1)$. This was shown to be true by the author~\cite{P18}. \\

We argue that this is an interesting regime for random regular hypergraphs, as this kind of discrepancy bound is not implied by the Beck-Fiala conjecture. The case where $n = \Omega(m \log m)$, is of particular interest, since we believe there is a phase transition for constant discrepancy at this point. On the one hand, we do not know if the discrepancy bound given by Corollary~\ref{cor:ranreg} is the truth, and on the other hand, we do not know if random regular hypergraphs with, for example, $n = \Theta(m^{1.5})$ almost surely has discrepancy $O(1)$. We conclude with a conjecture, building on an open problem (open problem $1$) in~\cite{FS18}:

\begin{conjecture}
There is an absolute constant $K >0$ such that the following holds. Let $t>0$ be any integer and $\mathcal{H}$ be a random $t$-regular hypergraph on $n$ vertices and $K\frac{n}{\log n}$ edges. Then with high probability,
\[
\disc(\mathcal{H}) = O(1).
\]
\end{conjecture}

\paragraph{Acknowledgements} I would like to thank Jeff Kahn for suggesting the problem, Huseyin Acan and Cole Franks, and Shachar Lovett for the extremely helpful discussions, and the anonymous ICALP referees for the numerous comments on a previous version.

\bibliographystyle{alpha}
\bibliography{references.bib}

\appendix

\section{Appendix}

\subsection{A martingale inequality}

We will state a martingale inequality that we will use in the proof of Theorem~\ref{thm:norm}. A sequence of random variables $X_0,X_1,\ldots, X_n$ martingale with respect to another sequence of random variables $Z_0, Z_1,\ldots, Z_n$ such that for all $i \in [n-1]$, we have $X_i = f_i(Z_1,\ldots Z_i)$ for some function $f_i$, and $\E[X_{i+1}|Z_i,\ldots, Z_1] = X_i$.\\

A martingale is said to have the $C$-bounded difference property if $|X_{i+1} - X_{i}| \leq C$.

The variance of a martingale is the quantity:
\[
\sigma^2 = \sum_{i \in [n-1]} \sup_{(Z_1,\ldots, Z_i)} \E[(X_{i+1} - X_i)^2| Z_1,\ldots, Z_i].
\]

We get good large deviation inequalities for martingales with bounded differences and variances (see, for example,~\cite{CL06}, Theorem 6.3 and Theorem 6.5). For a martingale $X_0, X_1,\ldots, X_n$ with respect to $Z_0, Z_1,\ldots, Z_n$, with the $C$-bounded difference property and variance $\sigma^2$, we have
\begin{equation}
\label{ineq:BdVar}
\mathbb{P}(|X_n - X_0| \geq \lambda ) \leq e^{-\frac{t^2}{2(\sigma^2 + C\lambda/3)}}.
\end{equation}

\subsection{Proof of Theorem~\ref{thm:norm}}
\label{sec:norm}

We shall now prove Theorem~\ref{thm:norm}. Recall that we only need to prove the case where $m = n$, As mentioned before, we adapt the proof technique of Kahn and Szemer\'{e}di for our random model (also see~\cite{BFSU98}). We have that the regularity is $t \ll m^{1/2}$. \\

We shall prove that for every $x$, and $y$ such that $\|x\| = \|y\| = 1$ and $x \perp \overline{1}$, we have that $|y^tMx| \leq O(\sqrt{t})$. First, we `discretize' our problem by restricting $x$ to belong to the $\epsilon$-net
\[
T := \left\{x \in \left( \frac{\epsilon}{\sqrt{m}}\mathbb{Z}\right)^m ~|~\|x\| \leq 1 \text{ and } x \perp \overline{1}\right\}
\]
and $y$ belonging to 
\[
T' := \left\{y \in \left( \frac{\epsilon}{\sqrt{m}}\mathbb{Z}\right)^m ~|~ \|y\| \leq 1 \right\}
\]
for a small enough constant $\epsilon$. \\

\begin{claim}[\cite{FKS}, Proposition 2.1)]
If for every $x \in T$, and $y \in T'$, we have that $\|y^tMx\| \leq \alpha$, then we have that for every $z \in \R^m$ such that $\|z\| = 1$, we have that $\|Mz\| \leq (1 - 3\epsilon)^{-1}\alpha$.
\end{claim}

\begin{proof}
Let $z = \operatorname{argmax}_{\|z\| = 1}\|Mz\|$. We shall use the fact that there are $x \in T$, and $y \in T'$ such that $\|x - z\| \leq \epsilon$, and $\left\|y - \frac{Mz}{\|Mz\|} \right\| \leq \epsilon$. With this in mind, we have:
\begin{align*}
\|Mz\| & = \left\langle \frac{Mz}{\|Mz\|}, Mz \right\rangle  = \langle y + w_1, M(x+w_2) \rangle \\
& = y^tMx + \langle w_1, Mx \rangle + \langle y,Mw_2 \rangle + \langle w_1 ,Mw_2\rangle.
\end{align*}

Where $|w_1|,|w_2| \leq \epsilon$. We note that each of the terms $\langle w_1, Mx \rangle$ and $\langle y,Mw_2 \rangle$, and $\langle w_1 ,Mw_2\rangle$ are upper bounded by $\epsilon \|Mz\|$, and $\langle w_1 ,Mw_2 \rangle \leq \epsilon^2 \|Mz\|$. Combining this, and using the fact that $\epsilon^2 \leq \epsilon$, we have
\[
\|Mz\| \leq (1 - 3\epsilon)^{-1}y^tMx \leq (1 - 3\epsilon)^{-1}\alpha.
\]
\end{proof}

So now, will need to only union bound over $T \cup T'$. It is not hard to see that each of these has size at most $|T|,|T'| \leq \left(\frac{C_v}{\epsilon}\right)^m$ for some absolute constant $C_v$.\\

Indeed, we have:
\begin{align*}
|T| & \leq \left(\frac{\sqrt{m}}{\epsilon}\right)^m\operatorname{Vol}\left\{x \in \R^m ~|~ \|x\| \leq 1+ \epsilon\right\}\\
&  \leq \left(\frac{\sqrt{m}}{\epsilon}\right)^m \cdot \frac{1}{\sqrt{\pi m}} \left( \frac{2 \pi e}{m}\right)^{m/2}(1+\epsilon)^m  \\
&\leq \left(\frac{C_v}{\epsilon}\right)^m
\end{align*}

for some constant $C_v$.\\

We split the pairs $[m] \times [m] = L \cup \overline{L}$ where $L := \{(u,v) ~|~ |x_uy_v| \geq \sqrt{t}/m\}$, which we will call `large entries' and write our quantity of interest:
\begin{align*}
\sum_{(u,v) \in [m] \times [m]} x_uM_{u,v}y_v = \sum_{(u,v) \in L}x_uM_{u,v}y_v + \sum_{(u,v) \in \overline{L}}x_uM_{u,v}y_v.
\end{align*}

\textbf{For the large entries}: For a set of vertices $A \subset [m]$ and a set of edges $B \subset [m]$, let us denote $I(A,B)$ to be the number of vertex-edge incidences in $A$ and $B$. Let us use $\mu(A,B) := \E[|I(A,B)|]$.

\begin{lemma}
\label{lem:dense}
There is a constant $C$ such that, for every set $A$ of vertices and every set $B$ of hyperedges where $|A| \leq |B|$, we have that with probability at least $1 -m^{- \Omega(1)}$, $I := |I(A,B)|$ and $\mu := \mu(A,B)$ satisfy at least one of the following:
\begin{enumerate}
\item $I \leq C \mu$
\item $I \log \left(I/\mu \right) \leq C |B| \log \left( m / |B| \right)$.
\end{enumerate}
\end{lemma}

This lemma is sufficient to show that the large pairs do not contribute too much, as shown by the following lemma, which is the main part of the proof of Kahn and Szemer\'{e}di.

\begin{lemma}[\cite{FKS}, Lemma 2.6,~\cite{BFSU98}, Lemma 17] 
If the conditions given in Lemma~\ref{lem:dense} are satisfied, then $\sum_{(u,v) \in L} \left|x_u M_{u,v}y_v \right| = O(\sqrt{t})$ for all $x,y \in T$.
\end{lemma}

Notice that since we are bounding $\sum_{(u,v) \in L} \left|x_u M_{u,v}y_v \right| = O(\sqrt{t})$, which is much stronger than what we really need, it is okay to consider both $x$ and $y$ from $T$.

\begin{proof}[Proof of Lemma~\ref{lem:dense}]
First, we observe that it is enough to consider $|B| \leq m/2$, since otherwise, $|I(A,B)| \leq d|A| \leq 2\mu(A,B)$. Let $\mathcal{B}_i(a,b)$ denote the event that there is an $A$ of size $a$ and a $B$ of size $b$ which do not satisfy either of the conditions (with a fixd constant $C$ to be specified later) and $|I(A,B)| = i$. Before, we prove the lemma, let us make some observations, which (in hindsight) help us compute the probabilities much easier. Let $A$ be a set of $a$ vertices and $B$ be a collection of $b$ edges, such that $a \leq b \leq m/2$. \\

The point here is that we basically want to evaluate the sum:
\begin{align*}
\mathbb{P}\left(\bigcup_{a,b,i}\mathcal{B}_i(a,b) \right) & \leq \sum_i \mathbb{P}\left(\bigcup_{a,b}\mathcal{B}_i(a,b)\right) \\
& = \sum_{i \leq \log ^2m}\mathbb{P}\left(\bigcup_{a,b}\mathcal{B}_i(a,b)\right) + \sum_{i \geq \log^2 m}\mathbb{P}\left(\bigcup_{a,b}\mathcal{B}_i(a,b)\right).
\end{align*}
The first observation is that every term in the second sum is small. Towards this, we have the straightforward claim.
\begin{claim}
\label{claim:fixed}
For a set of vertices $A$ and edges $B$ and a set of possible incidences $J \subset A \times B$, we have that $\mathbb{P}(I(A,B) = J) \leq \left( \frac{2t}{m}\right)^{|J|}$.
\end{claim}

\begin{proof}
W.L.O.G, let $A = \{1,\ldots, a\}$, and for $i \in A$, let $t_i = I(\{i\},B)$. We have that:
\begin{align*}
\mathbb{P}(I(A,B) = J)  = \prod_{i \in A}\frac{\binom{m-b}{t- t_i}}{\binom{m}{t}}  \leq  \prod_{i \in A} 2 \frac{(m - b)^{t - t_i}}{(t - t_i)!} \frac{t!}{m^t} \leq \prod_{i \in A}\left(\frac{2t}{m}\right)^{t_i} \leq \left(\frac{2t}{m}\right)^{|J|}.
\end{align*}
\end{proof}

Here, the first inequality uses the fact that $t = o(\sqrt{m})$. Therefore, we have:

\begin{align*}
\mathbb{P}\left(\mathcal{B}_i(a,b)\right) \leq \binom{m}{a}\binom{m}{b}\binom{ab}{i}\left(\frac{2t}{m}\right)^i  \leq \binom{m}{b}^2\left(e\frac{abt}{mi}\right)^i \leq \binom{m}{b}^2 \left(\frac{\mu}{i}\right)^i  (e)^i.
\end{align*}

If $i \geq 2e \mu$ and $i \geq \log^2 m$, this probability is at most $2^{2m} \cdot 2^{- \log^2 m} \ll m^{-\Omega(\log m)}$. Thus 
\[
\sum_{i \geq \log^2 m} \mathbb{P}\left(\bigcup_{a,b}\mathcal{B}_{i}(a,b)\right) \leq \sum_{a,b}\sum_{i \geq \log^2 m} \mathbb{P}\left(\mathcal{B}_{i}(a,b)\right) \leq m^{- \Omega(\log m)}.
\]

It remains to deal with the sum $\sum_{i \leq \log ^2m}\Pr\left(\bigcup_{a,b}\mathcal{B}_i(a,b)\right)$. For these summands, we have that if $|I(A,B)| \leq \log^2m$ and $I \log (I/\mu) > Cb \log(m/b)$, then
\[
I\log m \geq I \log (I/\mu) > Cb \log(m/b) \geq Cb.
\] 
 and so $Cb \leq  \log^3 m$. The first inequality above comes from the observation that $I \leq ab$ and so $I / \mu \leq m/t \leq m$. Now, using that $I \log m \geq Cb \log (m/ \log^3 m)$, we have that $I \geq Cb/2$. 

Therefore, we only need to evaluate the sum:
\begin{align*}
\sum_{i = Cb/2}^{\log^2 m}\mathbb{P}\left(\mathcal{B}_i(a,b)\right) & \leq \binom{m}{a}\binom{m}{b}\sum_{i = Cb/2}^{\log^2 m}\binom{ab}{i}\left( \frac{10et}{m}\right)^i  \leq \binom{m}{b}^2\sum_{i = Cb/2}^{\log^2 m} \left( \frac{10e^2abt}{im}\right)^i \\
& \leq \log^2m \left(\frac{em}{b}\right)^{2b}\left(\frac{20e^2at}{Cm}\right)^{Cb/2} \\
& = m^{2b - Cb/2} b^{-2b}a^{Cb/2}t^{Cb/2}(20e^2)^{2b} \\
& \leq m^{2b - Cb/4} b^{Cb/2 - 2b} (20e^2)^{2b}\\
& = m^{- \Omega(b)}.
\end{align*}

We have used the fact that $t = o(\sqrt{m})$, $b \geq a$ and $b\leq \log^3 m$. Thus union bounding over $\log^3 m$ many values of $a$ and $b$, we have $\sum_{a,b \leq \log^3 m}\sum_{i \leq \log^2 m} \mathbb{P}\left(\bigcup_{a,b}\mathcal{B}_{i}(a,b)\right) = m^{- \Omega(1)}$.
\end{proof}

\textbf{For the small entries:} Bounding the contribution from the small entries is much easier. The analysis given here is slightly different to the one given in~\cite{FKS} and~\cite{BFSU98}. However, it does not make much of a difference, and is still, essentially, the same large deviation inequality. We will first compute the expected value of the quantity of interest using the following claim:

\begin{claim}
\label{claim:sumbound}
We have that:
\[
\left| \sum_{(u,v) \in \overline{L}} x_u y_v\right| \leq \frac{m}{\sqrt{t}}.
\]
\end{claim}

\begin{proof}
Since $\sum x_i = 0$, we have $\left( \sum x_i\right) \left(\sum y_i \right) = \sum_{(u,v) \in L} x_uy_v + \sum_{(u,v) \in \overline{L}}x_uy_v = 0$
or 
\[
\left| \sum_{(u,v) \in \overline{L}} x_u y_v\right|  = \left| \sum_{(u,v) \in L} x_u y_v\right|.
\]
To bound this, we note that 
\begin{align*}
1 & = \left(\sum x_u^2\right) \left( \sum y_u^2 \right) \geq \sum_{(u,v) \in L} x_u^2y_v^2  \geq \frac{\sqrt{t}}{m}\left| \sum_{(u,v) \in L}x_uy_v \right| =  \frac{\sqrt{t}}{m}\left| \sum_{(u,v) \in \overline{L}}x_uy_v \right|.
\end{align*}

which gives us what we want.
\end{proof}

Given Claim~\ref{claim:sumbound} above, we can easily compute the expectation:
\[
\E\left[ \sum_{(u,v) \in \overline{L}}x_uM_{u,v}y_v \right] = \frac{t}{m} \sum_{(u,v) \in \overline{L}}x_uy_v \in [- \sqrt{t}, \sqrt{t}].
\]
\begin{claim}
\label{claim:smallpairs}
We have that with high probability, $\sum_{(u,v) \in \overline{L}}x_uM_{u,v}y_v = O(\sqrt{t})$.

\end{claim}

\begin{proof}
We set up a martingale and use the method of bounded variances. Let us write the quantity that we wish to estimate as 
\[
X := \sum_{(u,v) \in B}x_uM_{u,v}y_v.
\]

We imagine $M$ being sampled one column at a time, and in each column, $t$ entries are sampled. For column $i$, let us denote these by $e_{i,1}, \ldots, e_{i,t}$. Clearly, $X = X(e_{1,1},\ldots, e_{m,t})$. Denote $X_{i,j} := \E[X | e_{1,1} ,\ldots, e_{i,j}]$. For distinct $k,k' \in [m]$, it is easy to see that we have the `Lipschitz property':
\[
|\E[X | e_{1,1},\ldots, e_{i,j-1}, e_{i,j} = k] - \E[X | e_{1,1},\ldots, e_{i,j-1}, e_{i,j} = k']| \leq |x_iy_k| + |x_iy_{k'}|.
\]

Therefore, we have a bounded difference property on $|X_{i,j} - X_{i,j-1}|$ as follows:
\begin{align*}
|X_{i,j} - X_{i,j-1}| & = \Bigg|\E[X | e_{1,1},\ldots, e_{i,j-1}, e_{i,j}]  \\
&~~~~ -  \frac{1}{m - j + 1}\sum_{k' \in [m] \setminus \{e_{i,1},\ldots,e_{i,j-1}\}}\E[X | e_{1,1},\ldots, e_{i,j-1}, e_{i,j} = k'] \Bigg| \\
& \leq |x_{e_j}||y_i| + \frac{1}{m - j + 1}\sum_{k' \in [m] \setminus \{e_{i,1},\ldots,e_{i,j-1}\}}\mathbbm{1}[(k',i) \in \overline{L}]|x_{k'}y_i|\\
\end{align*}

We will use that the above quantity is bounded by $\frac{2\sqrt{t}}{m}$ since we only consider $|x_{e_j}y_i|$ where $(e_j,i) \in \overline{L}$. However, another way to upper bound the above is by using 

\begin{align*}
  &\frac{1}{m - j + 1}\sum_{k' \in [m] \setminus \{e_{i,1},\ldots,e_{i,j-1}\}}\mathbbm{1}[(k',i) \in \overline{L}]|x_{k'}y_i| \\
  &\leq \frac{1}{m - j + 1}\sum_{k' \in [m] \setminus \{e_{i,1},\ldots,e_{i,j-1}\}}|x_{k'}y_i|\\
  &\leq  \frac{|y_i|}{n - j + 1}\sum_{k' \in [m]}|x_{k'}| \\
  &\leq \frac{2|y_i|}{\sqrt{m}}.
\end{align*}

Using this, we now compute the variance of the martingale:
\begin{align*}
\Var(X_{i,j} - X_{i,j-1} | e_{1,1},\ldots, e_{i,j-1}) & \leq \frac{1}{m - j + 1} \sum_{k \in [m]} \left(|x_ky_i| + \frac{2|y_i|}{\sqrt{m}}\right)^2 \\
& \leq \frac{2}{m - j + 1} \sum_{k \in [m]}\left(|x_ky_i|^2 + \frac{4y_i^2}{m} \right) \\
& \leq \frac{10y_i^2}{m - j + 1}.
\end{align*}

Where the last inequality uses that $\sum_k x_k^2 \leq 1$. Therefore, the variance of the martingale is at most$t \cdot \frac{10}{m - t}\sum_i y_i^2 \leq \frac{20t}{m} =: \sigma^2$. This is because $\sum_i y_i^2 \leq 1$. Therefore, by the bounded variance martingale inequality~(\ref{ineq:BdVar}), using $|X_i - X_{i-1}|\leq \frac{2\sqrt{t}}{m} =: C$:

\begin{align*}
\mathbb{P}(X \geq (D+1)\sqrt{t}) & \leq \exp \left\{ - \frac{D^2t}{2\sigma^2 + tC/3} \right\} \leq \exp \left\{ -\frac{D^2t}{\frac{40t}{m} + \frac{2t}{3m}}\right\}  \leq \exp\left\{ -\Omega(D^2m) \right\}.
\end{align*}

For a large enough constant $D$, this lets us union bound over all $x,y \in T$, whose number can be bounded by $\left( \frac{C_v}{\epsilon}\right)^m$.

\end{proof}

\end{document}